\documentclass[10pt,twoside,a4paper]{amsart}

\usepackage{amsfonts,amsmath,amsthm}
\usepackage{hyperref}

\newtheorem{thm}{Theorem}[section]

\newtheorem{cor}[thm]{Corollary}

\theoremstyle{remark}
\newtheorem{remark}[thm]{Remark}

\theoremstyle{ass}
\newtheorem{ass}[thm]{Assumption}

\newcommand{\R}{\mathbb{R}}
\newcommand{\Rd}{\mathbb{R}^{d}}
\newcommand{\D}{\nabla_{A}}
\newcommand{\h}{\mathcal{H}}

\usepackage[margin=1in, top=1.5in, bottom=1.5in]{geometry}

\numberwithin{equation}{section}

\title[Unique continuation for magnetic Schr\"odinger]{Unique continuation for magnetic Schr\"odinger operator with singular potentials}
\author[N. Arrizabalaga, M. Zubeldia]{Naiara Arrizabalaga$^{1}$, Miren Zubeldia$^{2}$}

\date{October, 2013}


\address{$^{1}$ Departamento de Matem\'aticas, Universidad del Pa\'is Vasco, Apartado 644, 48080, Bilbao, Spain.}
\email{naiara.arrizabalaga@ehu.es}

\address{$^{2}$ Department of Mathematics and Statistics, University of Helsinki, P.O. Box 68, FI00014, Helsinki, Finland.}
\email{miren.zubeldia@helsinki.fi}

\begin{document}

\begin{abstract}
In this paper we study unique continuation theorems for magnetic Schr\"odinger equation via Carleman estimates. We use integration by parts techniques in order to show these estimates. We consider electric and magnetic potentials with strong singularities at the origin and some decay at infinity. 
\end{abstract}

\maketitle

\section{Introduction}

Let us consider  the magnetic Schr\"odinger operator
\begin{equation}
H_{A, V} = \D^2 + V
\end{equation}
with
$$
\D := \nabla +  iA.
$$
Here $A:\Rd \to \Rd$ is the magnetic vector potential and $V:\Rd \to \R$ is the electric scalar potential. We denote by $H_{A}$ the particular case $V=0$. Following the notation of \cite{BVZ}, we define the magnetic field corresponding to the magnetic potential $A$ by the $d\times d$ anti-symmetric matrix given by
\begin{equation}
B = (DA) - (DA)^{t}, \quad \quad B_{kj} = \left(\frac{\partial A_{k}}{\partial x_{j}} - \frac{\partial A_{j}}{\partial x_{k}} \right) \quad k,j=1, \ldots, d.
\end{equation}
In dimension $d=3$, $B$ is identified by the vector field $curl \, A$ via the vector product
\begin{equation}
Bv = curl \, A \times v, \quad \quad \forall v \in \R^3.
\end{equation}
Of particular interest is the tangential part or the trapping component of $B$ defined as
\begin{equation}
B_{\tau}(x) = \frac{x}{|x|}B(x), \quad \quad (B_{\tau})_{j} = \sum_{k=1}^d \frac{x_k}{|x|}B_{kj}.
\end{equation}
Observe that $B_{\tau} \cdot x = 0$ for any $d\geq 2$ and, therefore, $B_{\tau}$ is a tangential vector field in any dimension.

In this paper we are interested in proving unique continuation theorems for the operator
\begin{equation}
H_{A,V}^{\lambda} = H_{A,V} + \lambda \quad \quad \lambda \geq 0,
\end{equation}
with singular electric and magnetic potentials. Let $u$ be a solution to the equation
\begin{equation}\label{ekuazioa}
H_{A,V}^{\lambda} u = 0
\end{equation}
such that decays exponentially or more rapidly than any power of $|x|^{-1}$ at infinity. Then, under suitable assumptions on $V$ and $B_{\tau}$, we shall prove that $u=0$ in $\Rd$, $d\geq 3$.

We point out that we are able to consider magnetic potentials such that $B_{\tau}$ have singularities at the origin as $\frac{\nu_1}{|x|^2}$ for $\nu_1 >0$ small enough. In addition, our results are also valid for singular magnetic potentials $A$ such that $B_{\tau}=0$. The well known Aharonov-Bohm potential,
$$
A= C\left(\frac{-y}{x^2 + y^2}, \frac{x}{x^2 + y^2}, 0 \right),
$$
or magnetic potentials $A$ homogeneous of degree $-1$ that satisfy the transversality condition, i.e.
$$
\lambda A(\lambda x)= A(x) \quad \quad \text{and} \quad \quad x\cdot A(x)=0
$$
for all $\lambda >0$, are some relevant examples that produce non-trapping magnetic fields ($B_{\tau}\equiv 0$).

This kind of singular magnetic potentials have been widely studied by several authors in the last years. It has been showed that small $B_\tau$ plays an important role in the study of the dispersive estimates for magnetic Schr\"odinger equation (see for example \cite{PLNL, FFFP, FG, FV}). Furthermore, uniform a priori estimates for the electromagnetic Helmholtz equation with singular small $B_{\tau}$ or with $B_{\tau}=0$ have been proved in \cite{BFRV, BVZ, F}  and \cite{Z}, among others. One of the important applications of these estimates is the so called \emph{limiting absorption principle}, for which the resolvent operator $-H_{A,V}$ can be extended, on the positive real line, to a bounded operator between weighted $L^2$-spaces. Moreover, starting by these estimates, it is possible to obtain the appropriate Sommerfeld radiation condition implying uniqueness of solution, as it was recently showed by Zubeldia in \cite{Z}. However, in order to prove the uniqueness the author uses a unique continuation result \cite{Re} assuming extra conditions in the magnetic potential $A$ to assure that this result is applicable.

As far as we know, unique continuation problems for magnetic Schr\"odinger operators with the singularities mentioned above can not be found in the literature. Singular magnetic fields of Kato class are considered by Kurata in \cite{Ku1} and \cite{Ku2}. More recently, Xiaojun Lu in \cite{xl} studies the unique continuation problem assuming some smoothness condition on the magnetic potential, more concretely, $A \in C^1$. Whereas, there are more works for singular electric potentials, as for example, \cite{Pa}, where Pan proves several unique continuation theorems for the Schr\"odinger operator with critical singularities on the electric potential as $|x|^{-2}$ in the unit ball. Furthermore, there are several works concerning with general elliptic operators. See \cite{BKRS, GL1, GL2, KT, Re, So} and \cite{Wo1}, among  others.

In this work we prove unique continuation results for Schr\"odinger operators with singular electric and magnetic potentials that we specify below. These results are shown by using Carleman type estimates as a main tool. We show those estimates for the operator  $H_{A,V}^{\lambda}$ and for the particular case $V=0$ denoted by $H_{A}^{\lambda}$. Carleman estimates with two different weights are given for both operators. On the one hand, we show estimates with exponential weight of the type $e^{\tau|x|(1+\log |x|)}$ with $\tau >0$ which are used to show that if $u$ satisfies the equation $H_{A,V}^{\lambda}u = 0$ and decays exponentially at infinity, then $u$ vanishes identically in $\Rd$. On the other hand, Carleman estimates with polynomial weight $|x|^{\tau}$, $\tau >0$, are needed for solutions $u$ such that $|x|^{-m} u$ tends to zero as $|x|$ tends to infinity for every $m>0$.

The Carleman estimates are proved by using integration by parts. Thus in order to justify this argument we need some regularity on the solution $u$. For this purpose, we define the distorted Sobolev space $\h^s$ generated by the positive powers of the operator $H_{A,V}$ denoted by $H_{A,V}^s$ . More precisely,
\begin{equation}
\Vert f\Vert_{\h^s} := \Vert H_{A,V}^s f \Vert_{L^2}, \quad \quad s\geq 0.
\end{equation}
These spaces are the natural ones in which we work in the sequel. Note that if the Hamiltonian $H_{A,V}$ is self-adjoint and positive, the existence of $H_{A,V}^s$ is ensured via the spectral theorem. In addition, it can be shown that under certain conditions on the potentials $A,V$ the standard Sobolev space $H^s$ and $\h^s$ are equivalent. For more details we refer the reader to \cite{PLNL} (Theorem 1.2). This construction has also been used in \cite{F} (Remark 2.1) and in \cite{BFRV} (Remark 1.2).

Now we state the assumptions on the potentials considered in the unique continuation theorems. We will always assume $r_0\geq 1$. 

\begin{ass}\label{ass1}
For $d\geq3$,
\begin{equation*}\label{btau}
|B_{\tau}| \leq\left\{
\begin{split}
&\frac{\nu_1}{r^2},\quad r\leq r_0,\\
&\frac{\nu_2}{r},\quad r>r_0,
\end{split}
\right.
\qquad
(\partial_rV)_- \leq\left\{
\begin{split}
&\frac{\mu_1}{r^3},\quad r\leq r_0,\\
&\frac{\mu_2}{r},\quad r>r_0,
\end{split}
\right.
\end{equation*}
such that $ \nu_1^2+2\mu_1 < 2(d-1)(d-2)$ and $\frac{2}{3}\nu_2^2+\mu_2 < 1$. \end{ass}

\begin{ass}\label{ass2}
For $d\geq3$,
\begin{equation*}\label{btau}
|B_{\tau}| \leq\left\{
\begin{split}
&\frac{\nu_1}{r^2},\quad r\leq r_0,\\
&\frac{c}{r^{1+\mu}},\quad r>r_0,
\end{split}
\right.
\quad
(\partial_rV)_-\leq\left\{
\begin{split}
&\frac{\mu_1}{r^3},\quad r\leq r_0,\\
&\frac{c}{r^{1+\mu}},\quad r>r_0,
\end{split}
\right.
\end{equation*}
such that $ \nu_1^2+2\mu_1 < 2(d-1)(d-2)$ and $c,\mu>0$. 
\end{ass}

Here $(\partial_rV)_-$ denotes the negative part of $\partial_r V$. Moreover, $V=V_+-V_-$, where $V_{\pm}$ denotes the positive and negative parts of $V$, respectively. We can now formulate our main results.

\begin{thm}\label{UCP} 
Let $d\geq3$, $\lambda \geq 0$. Assume that the potentials satisfy one of the following hypotheses:

\begin{itemize}
\item[(i)] Assumption \ref{ass1} or \ref{ass2}. In addition, for $r>r_0$ $V$ is such that
\begin{equation}\label{vbarrixa}
V_{-} \leq \tau \beta \left(\log r +2 \right)\left[ \tau\beta\left(\log r +2 \right) +\frac{(d-1)}{r}\right] + \frac{\tau\beta}{2r} + V_+,
\end{equation}
where $\tau\beta > \frac{d-2}{r_0 (2+\log r_0)}$.
\item[(ii)] Let $B_{\tau}$ as in  Assumption \ref{ass1} or \ref{ass2} and, for $\gamma_i>0$ $i=1,2$, $V$ holds
\begin{equation}\label{Vucp}
|V| \leq\left\{
\begin{split}
&\frac{\gamma_1}{r^{3/2}},\quad r\leq r_0,\\
&\frac{\gamma_2(log r + 1)^{1/2}}{r^{1/2}}, \quad r>r_0.
\end{split}
\right.
\end{equation}
\end{itemize}
If $u \in \h^1$ is a solution of (\ref{ekuazioa}) for some $\varepsilon >0$ and satisfies 
\begin{align}\label{exp}
\int_{|x|>R} (|u|^2 + |\D u|^2) = O\left(e^{-mR(1+\log R)} \right) \quad \quad (R\to \infty)
\end{align}
for every $m>0$, then $u = 0$.
\end{thm}

\begin{thm}\label{(SUCP)}
Let $d\geq 3$, $\lambda >0$. Assume that $B_{\tau} \equiv 0$ and $V$ holds
\begin{equation}\label{assuV} 
\int \left[ \partial_r (rV) + V \right]_{-} |v|^2 < \int \left[ \partial_{r}(rV) + V \right]_{+} |v|^2
\end{equation}
 or 
\begin{equation}\label{Vsucp}
|V| \leq \frac{C}{|x|},
\end{equation}
for some $C>0$. If $u \in \h^1$ is a solution of (\ref{ekuazioa}) for some $\varepsilon >0$ satisfying 
\begin{equation}\label{pol}
\int_{|x|> R} (|u|^2 + |\D u|^2) = O\left( R^{-m} \right) \quad \quad (R\to \infty)
\end{equation}
for all $m > 0$, then $u =0$.
\end{thm}

\begin{remark}
An example of electric potential $V$ satisfying condition (\ref{assuV}) is the following
\begin{equation}
V(x)= \frac{V_{\infty}\left(\frac{x}{|x|}\right)}{|x|}, \quad \quad V_{\infty} >0.
\end{equation}
In this case we have $\partial_{r}(rV)=0$ and $V_{-}=0$. 
\end{remark}

Conditions (\ref{exp}) and (\ref{pol}) in the above theorems are related to the Carleman estimates used in the proofs of the theorems. In fact, Theorem \ref{UCP} is proved by using the Carleman estimate with exponential weight, while Theorem \ref{(SUCP)} involves the polynomial one. The required assumptions on the tangential component of the magnetic field also concern with this issue. The hypothesis $B_{\tau}=0$ is necessary to show the corresponding Carleman estimate of Theorem \ref{(SUCP)}. We point out that in order to consider magnetic potentials such that $|B_{\tau}|$ has singularities, one needs to work with exponential type Carleman estimates and, as a consequence, with condition (\ref{exp}). On the other hand, the hypotheses on $V$ are different depending on whether the Carleman estimate involves the operator $H_{A,V}^\lambda$ or $H_A^\lambda$. We can consider more singular electric potentials when we show the Carleman estimate for the operator $H_{A,V}^{\lambda}$. That is, when the potential $V$ itself is included on the estimate. Finally, the fact that $\lambda = 0$ is not considered in Theorem \ref{(SUCP)} is again another consequence of the polynomial Carleman estimate; assumption $\lambda >0$ is essential for that.

Section \ref{section2} is devoted to the proof of the Carleman estimates. Whereas the proofs of Theorem \ref{UCP} and Theorem \ref{(SUCP)} can be found in Section \ref{uc}.

\subsection*{Notation}
For the sequel, $C>0$ denotes a constant which may change its value at different occurrences. Throughout this paper $|x|$ and $r$ are used indistincly. Moreover, to shorten notation, we write $\int$ for the integral in the whole $\Rd$ with respect to the Lebesgue measure $dx$. We consider the outward normal vector of $r\leq r_0$  as $\frac{x}{|x|}$ and the one of $r> r_0$ as $-\frac{x}{|x|}$.

\section{Carleman estimate}\label{section2}

In this section we prove Carleman estimates with exponential and polynomial weights by using integration by parts techniques.

\subsection{Exponential weight}

We give a new Carleman estimate for the magnetic Schr\"odinger operator with critical singularities on the potentials. We first introduce the function $\varphi(x)=\varphi(|x|)>0$ which is the one that determines the weight. For $\beta > 0$ and $r_0 \geq 1$, it is given by
\begin{equation}\label{fi}
\varphi(x):=\left\{
\begin{split}
&\left(\beta+\frac{\beta}{2}\log r_0\right) |x|,\quad |x|\leq r_0,\\
&\frac{\beta}{2}|x|(\log |x|+1)+\frac{\beta}{2},\quad |x|>r_0,
\end{split}
\right.
\end{equation}

Now we are ready to state the result.

\begin{thm}\label{procarleman}
Let $d\geq 3, \, \lambda\geq 0,\, \tau>0$ and $\varphi$ defined in (\ref{fi}). Under Assumption \ref{ass2}  with $r_0$ large enough or Assumption \ref{ass1}, then for all $u\in C_0^{\infty}(\R^d)$
\begin{align}\label{carleman}
&\tau\int e^{2\tau \varphi}\frac{|u|^2}{|x|^3}+\tau\int_{r > r_0} e^{2\tau \varphi}\frac{\log |x|}{|x|^3}|u|^2 + \tau\beta\int_{r>r_0} e^{2\tau\varphi} |\D u|^{2} \\
&+\tau^2(1+\tau)\int_{r>r_0} e^{2\tau \varphi}(\log |x|+1)\frac{|u|^2}{|x|}+\tau\lambda\int_{r>r_0} e^{2\tau \varphi}\frac{|u|^2}{|x|}\notag\\ 
&+\frac{\tau\beta}{r_0}\Re \int_{r=r_0}e^{2\tau \varphi}\D^ru \bar{u}-\frac{\tau\beta(d-2)}{2r_0^2}\int_{r=r_0}e^{2\tau \varphi}|u|^2\notag\\
& \leq C \int e^{2\tau \varphi}|\D^2 u +V u +\lambda u|^2, \notag
\end{align}
where $C>0$ is independent of $\tau$.
\end{thm}

\begin{proof}

Let $v= e^{\tau\varphi}u$. Then it follows that
\begin{align}\label{equ1}
e^{\tau \varphi} (\D^2 u +V u + \lambda u) &= \D^2 v +V v + \lambda v + \tau^2 (\varphi')^2v\\
& - \frac{2\tau}{|x|}\left(\varphi'x\cdot \D v + \frac{\varphi''|x|+(d-1)\varphi'}{2}v \right)\notag
\end{align}
where $\varphi'$ denotes the radial derivative of $\varphi$.
Set
\begin{align}
& L_1 v = (\D^2  +V+ \lambda  + \tau^2(\varphi')^2) v \notag\\
& L_2  v= - \frac{2\tau}{|x|}\left(\varphi'x\cdot \D v + \frac{\varphi''|x|+(d-1)\varphi'}{2}v \right).\notag
\end{align}
Note that $L_1$ and $L_2$ are symmetric and skew-symmetric operators, respectively. 

Let $M$ denote the integral on the right-hand side of the inequality (\ref{carleman}). Then we have
\begin{align}
M & := \int e^{2\tau \varphi}|\D^2 u +V u +\lambda u|^2 \notag \\
& = \int |L_1 v + L_2 v|^2 \notag \\
& = \int |L_1 v|^2 + \int |L_ 2 v|^2 + 2\Re \int L_1 v \overline{L_2 v}. \notag
\end{align}
In particular, it yields
\begin{align}\label{mre}
2\Re \int L_1 v \overline{L_2 v} \leq M.
\end{align}

We next compute the term $\Re \int L_1 v \overline{L_2 v}$, which can be written as 
\begin{align}\label{Re}
\Re \int L_1 v \overline{L_2 v} = &-2\tau\Re\int \D^2 v \varphi' \frac{x}{|x|}\cdot\overline{\D v} - \tau\Re\int \left(\varphi''+\frac{\varphi'(d-1)}{|x|}\right)\D^2 v\bar{v}\\
& - 2\tau\lambda\Re\int \varphi'\frac{x}{|x|}\cdot\overline{\D v} v -\tau\lambda\int  \left(\varphi''+\frac{\varphi'(d-1)}{|x|}\right)|v|^2\notag\\
& -2\tau^3 \Re \int (\varphi')^3 \frac{x}{|x|}\cdot \overline{\D v}v -\tau^3\int (\varphi')^2 \left(\varphi''+\frac{\varphi'(d-1)}{|x|}\right)|v|^2\notag\\
& -2\tau\Re \int \varphi' V \frac{x}{|x|}\cdot \overline{\D v}v-\tau\int V \left(\varphi''+\frac{\varphi'(d-1)}{|x|}\right)|v|^2.\notag
\end{align}

We start studying the $\lambda$ terms. Since $\Re \overline{\D v} v = \Re \D v \bar{v}$, by integration by parts it is easy to check that
\begin{align*}
-2\tau\lambda \Re \int \varphi'\frac{x}{|x|}\cdot\overline{\D v}v &= \tau\lambda\int\left(\varphi''+\frac{\varphi'(d-1)}{|x|}\right)|v|^2\\
& +2\tau\lambda \Re \int_{r=r_0}(\varphi'(r_0^+)-\varphi'(r_0^-))|v|^2\notag.
\end{align*}
Here $r_0^{\pm}$ denote the limit of $r$ to $r_0$ from outside or inside the ball of radius $r_0$, respectively. Note that because of the definition of $\varphi$ some boundary integrals appear at $r=r_0$. This phenomena will happen in all the integration by parts of this paper. Moreover, since $\varphi'$ is continuous at $r_0$, i.e., $\varphi'(r_0^+)=\varphi'(r_0^-)$, it follows that 
\begin{align}\label{lambdaterm}
&-2\tau\lambda \Re \int \varphi'\frac{x}{|x|}\cdot\overline{\D v}v -\tau\lambda\int  \left(\varphi''+\frac{\varphi'(d-1)}{|x|}\right)|v|^2=0.
\end{align}

Similarly, we get
\begin{align}\label{tauterm}
&-2\tau^3 \Re \int (\varphi')^3\frac{x}{|x|}\cdot \overline{\D v} v - \tau^3\int (\varphi')^2 \left(\varphi''+\frac{\varphi'(d-1)}{|x|}\right)|v|^2\\
&=2\tau^3\int(\varphi')^2\varphi''|v|^2\notag.
\end{align}

Observe that by integration by parts we obtain
\begin{align*}
 &-2\tau\Re \int \varphi' V \frac{x}{|x|}\cdot \overline{\D v}v=\tau\int \varphi'(\partial_r V)|v|^2+\tau\int V \left(\varphi''+\frac{\varphi'(d-1)}{|x|}\right)|v|^2.
\end{align*}
Thus, the potential terms in (\ref{Re}) reduces to 
$$\tau\int \varphi'(\partial_r V) |v|^2.$$

We next analyze the $\D^2 v$ terms. Note that $\varphi'', \varphi'''$ are not continuous. Applying the same reasoning as above to the term involving $\D^2 v \bar{v}$, we obtain
\begin{align}\label{i1}
&-\tau \Re \int \left(\varphi''+\frac{\varphi'(d-1)}{|x|}\right)\D^2 v \bar{v} \\
 &= \tau \Re \int \left[\varphi'''+\left( \frac{\varphi''}{|x|}-\frac{\varphi'}{|x|^2}\right)(d-1) \right]\frac{x}{|x|} \cdot \D v \bar{v}\notag\\
 &+\tau\int \left(\varphi''+\frac{\varphi'(d-1)}{|x|}\right)|\D v|^2\notag \\
& + \tau \Re \int_{r=r_0}(\varphi''(r_0^+)-\varphi''(r_0^-))\frac{x}{|x|}\cdot\D v\bar{v}  \notag. 
\end{align}
Now, integrating by parts the first term on the right hand side of (\ref{i1}) we get that it is equal to
\begin{align}\label{i12}
&-\frac{\tau}{2}\int \left[\varphi^{IV}+\left(2\varphi'''+\varphi''\frac{d-3}{|x|}-\varphi'\frac{d-3}{|x|^2}\right)\frac{d-1}{|x|}\right]|v|^2\\ \notag
&-\frac{\tau}{2}\int \left[\varphi^{'''}+ \left(\frac{\varphi''}{|x|}-\frac{\varphi'}{|x|^2}\right) (d-1)\right]\frac{d-1}{|x|}|v|^2\\ \notag
&-\frac{\tau}{2}\int_{r=r_0} \left[\varphi'''(r_0^+)-\varphi'''(r_0^-)+\frac{(d-1)}{r_0}\left(\varphi''(r_0^+)-\varphi''(r_0^-)\right)\right]|v|^2.\notag
\end{align}
We introduce (\ref{i12}) in (\ref{i1}) and we get that 
\begin{align}\label{i13}
&-\tau \Re \int \left(\varphi''+\frac{\varphi'(d-1)}{|x|}\right)\D^2 v \bar{v} \\\notag
 &=-\frac{\tau}{2}\int \left[\varphi^{IV}+\left(3\varphi'''+\varphi''\frac{2(d-2)}{|x|}-\varphi'\frac{2(d-2)}{|x|^2}\right)\frac{d-1}{|x|}\right]|v|^2\\\notag
 &+\tau\int \left(\varphi''+\frac{\varphi'(d-1)}{|x|}\right)|\D v|^2\notag \\
 & +\tau \Re \int_{r=r_0}\left(\varphi''(r_0^+)-\varphi''(r_0^-)\right)\frac{x}{|x|}\cdot\D v\bar{v}  \notag\\
&-\frac{\tau}{2}\int_{r=r_0} \left[\varphi'''(r_0^+)-\varphi'''(r_0^-)+\frac{(d-1)}{r_0}\left(\varphi''(r_0^+)-\varphi''(r_0^-)\right)\right]|v|^2.\notag
\end{align}

To deal with the term involving $\D^2 v x\cdot\overline{\D v}$, it is convenient to first observe that for $j,k=1, \ldots, d$ we have
\begin{equation}\notag
\left(\frac{\partial}{\partial x_j}+i A_j \right)\left(\frac{\partial u}{\partial x_k} +i A_k u \right) = \left(\frac{\partial}{\partial x_k}+i A_k \right)\left(\frac{\partial u}{\partial x_j} +i A_j u \right) - i B_{jk}u,
\end{equation}
where recall that $B_{jk} = \frac{\partial A_j}{\partial x_k} - \frac{\partial A_k}{\partial x_j}$. As a consequence, one can conclude that
\begin{equation}\label{ondorioa}
\left(\frac{\partial}{\partial x_j}+i A_j \right)\overline{\left(\frac{\partial u}{\partial x_k} +i A_k u \right)} = \left(\frac{\partial}{\partial x_k}+i A_k \right)\overline{\left(\frac{\partial u}{\partial x_j} +i A_j u \right)} + i B_{jk}\bar{u}.
\end{equation}

Now we are ready to study the remaining term. We denote $\D^r u = \frac{x}{|x|}\cdot \D u$ and $\D^\bot u = \D u - \D^r u$. By writing
$$
\D^2 v \frac{x}{|x|}\cdot \overline{\D v} = \sum_{j,k=1}^d \left[\left(\frac{\partial}{\partial x_j} + iA_j \right)\left(\frac{\partial v}{\partial x_j}+iA_j v \right)\frac{x_k}{|x|}\left( \frac{\partial\bar{v}}{\partial x_k} -iA_k\bar{v} \right) \right],
$$ 
and using integration by parts, we get
\begin{align}\label{i2}
&-2\tau \Re \int \varphi'\D^2 v \frac{x}{|x|}\cdot \overline{\D v} = 2\tau \int \left(\varphi''-\frac{\varphi'}{|x|}\right)\left| \D^r v \right|^2 \\ \notag
&-2\tau \int_{r=r_0}(\varphi'(r_0^+)-\varphi'(r_0^-))\left|\D^r v \right|^2+  2\tau \int \varphi' |x|^{-1}|\D v|^2\\ \notag
&+ 2\tau \sum_{j,k=1}^d \Re \int\varphi' \left(\frac{\partial v}{\partial x_j} +iA_j v \right)\frac{x_k}{|x|}\left(\frac{\partial}{\partial x_j}+iA_j \right)\left(\frac{\partial \bar{v}}{\partial x_k}-iA_k \bar{v} \right).\notag\\
& \equiv I_1 + I_2 + I_3+I_4.\notag
\end{align}
By (\ref{ondorioa}) and using that $(B_\tau)_{j} = \sum_{k=1}^d \frac{x_k}{|x|}B_{jk}$, $B_\tau \cdot \D v = B_\tau \cdot \D^\bot v$, $\Re iz = -\Im z$ we get
\begin{align}\label{i3}
I_4 &= 2\tau \sum_{j,k=1}^d \Re \int \varphi'\left(\frac{\partial}{\partial x_k}+i A_k \right)\overline{\left(\frac{\partial v}{\partial x_j} +i A_j v \right)}\frac{x_k}{|x|} \left(\frac{\partial v}{\partial x_j} +iA_j v \right)\\
& -2\tau \Im \int  \varphi'B_\tau \cdot \D^\bot v \bar{v}\notag\\
& \equiv I_{41} + I_{42},\notag
\end{align} 
which at the same time, the integration by parts gives
\begin{align}\label{i4}
I_{41} = -\tau  \int \left(\varphi''+\varphi'\frac{d-1}{|x|}\right)|\D v|^{2}.
\end{align}

Hence, combining (\ref{i1})-(\ref{i4}) and using that  $|\D v|^2 - \left| \D^r v \right|^2 = |\D^\bot v|^2$, the terms involving $\D^2 v$ can be given by
\begin{align*}
&-\tau\int \left(\frac{\varphi^{IV}}{2} +\varphi'''\frac{(d-1)}{|x|}+\varphi''\frac{(d-1)(d-2)}{|x|^2}-\varphi' \frac{(d-1)(d-2)}{|x|^3}\right)|v|^2\\ \notag
&+ 2\tau\int \frac{\varphi'}{|x|}|\D^{\bot}v|^2+ 2\tau\int \varphi'' |\D^rv|^2- 2\tau \Im \int \varphi' B_\tau\cdot \D^{\bot}v \bar{v} \\\notag
& + \tau \Re \int_{r=r_0}\left(\varphi''(r_0^+)-\varphi''(r_0^-)\right)\frac{x}{|x|}\cdot\D v\bar{v}  \notag\\
&-\frac{\tau}{2}\int_{r=r_0} \left[\varphi'''(r_0^+)-\varphi'''(r_0^-)+\frac{(d-1)}{r_0}\left(\varphi''(r_0^+)-\varphi''(r_0^-)\right)\right]|v|^2.\notag
\end{align*}

As a consequence, combining the last computations with (\ref{mre}) and writing $\partial_r V=(\partial_r V)_+-(\partial_r V)_-$, we obtain

\begin{align}\label{c1}
&2\tau\int \left(\varphi' \frac{(d-1)(d-2)}{|x|^3}-\frac{\varphi^{IV}}{2} -\varphi'''\frac{(d-1)}{|x|}-\varphi''\frac{(d-1)(d-2)}{|x|^2}\right)|v|^2\\ \notag
&+ 4\tau\int \frac{\varphi'}{|x|}|\D^{\bot}v|^2+ 4\tau\int \varphi'' |\D^rv|^2+ 4\tau^3\int (\varphi')^2\varphi''|v|^2\\\notag
& -\frac{\tau\beta(d-2)}{2r_0^2}\int_{r=r_0} |v|^2 + 2\tau\int \varphi'(\partial_r V)_{+} |v|^2 \\ \notag
&\leq \int e^{2\tau \varphi}|\D^2 u +V u +\lambda u|^2 + 4\tau \Im \int \varphi' B_\tau\cdot \D^{\bot}v\bar{v} + 2\tau\int \varphi'(\partial_r V)_{-} |v|^2 \\\notag
& - \frac{\tau\beta}{r_0} \Re \int_{r=r_0}\D^r v\bar{v}.
\end{align}

Now replacing the values of the derivatives of $\varphi$, we get

\allowdisplaybreaks[4]
\begin{align}\label{c3}
&2\tau\beta\left(1+\frac{\log r_0}{2}\right)\left[(d-1)(d-2)\int_{r\leq r_0} \frac{|v|^2}{r^3}+2\int_{r\leq r_0} \frac{|\D^{\bot}v|^2}{r}\right] \\ \notag
&+\tau\beta(2+\log r_0)\int_{r\leq r_0} (\partial_r V)_+ |v|^2+\tau\beta\int_{r>r_0}(\log r+2)(\partial_r V)_+ |v|^2\\ \notag
&+2\tau\beta\int_{r>r_0}\frac{\log r}{r}|\D^{\bot}v|^2+4\tau\beta\int_{r>r_0}\frac{|\D^{\bot}v|^2}{r}+2\tau\beta\int_{r>r_0}\frac{|\D^rv|^2}{r}\\ \notag
&+\tau\beta[(d-1)^2-1]\int_{r>r_0}\frac{|v|^2}{r^3}+\tau\beta(d-1)(d-2)\int_{r>r_0}\frac{\log r}{r^3}|v|^2\notag\\
&+\frac{\tau^3\beta^3}{2}\int_{r>r_0}\frac{(\log r)^2}{r}|v|^2+2\tau^3\beta^3\int_{r>r_0}\frac{\log r}{r}|v|^2+2\tau^3\beta^3\int_{r>r_0}\frac{|v|^2}{r}\notag\\ 
&\leq \int e^{2\tau \varphi}|\D^2 u +V u +\lambda u|^2- \frac{\tau\beta}{r_0}\Re\int_{r=r_0} \D^r v\bar{v}+\frac{\tau\beta(d-2)}{2r_0^2}
\int_{r=r_0}|v|^2\notag\\ 
& + 2\tau\beta(\log r_0+2) \Im \int_{r\leq r_0} B_\tau\cdot \D^{\bot}v \bar{v}+2\tau\beta \Im \int_{r> r_0} (\log r+2)B_\tau\cdot \D^{\bot}v \bar{v} \notag\\
&+2\tau\beta(1+\frac{1}{2}\log r_0)\int_{r\leq r_0} (\partial_r V)_- |v|^2+\tau\beta\int_{r>r_0}\log r(\partial_r V)_- |v|^2\notag\\
&+2\tau\beta\int_{r>r_0}(\partial_r V)_- |v|^2.\notag
\end{align}

Until this moment the proof has been the same for both assumptions. However, from now on we will separate it for each case. \\

\textit{For Assumption \ref{ass1}. }
We would like to estimate the potential terms of the RHS of (\ref{c3}). By applying the Cauchy-Schwarz inequality to the magnetic field term in the RHS of (\ref{c3}) we have
\begin{align}\label{c4}
&2\tau\beta(\log r_0+2) \Im \int_{r\leq r_0} B_\tau\cdot \D^{\bot}v \bar{v}+2\tau\beta \Im \int_{r> r_0} (\log r+2)B_\tau\cdot \D^{\bot}v \bar{v}\\ \notag
&\leq 4\tau\beta(\frac{1}{2}\log r_0+1)\left(\int_{r\leq r_0}\frac{|\D^{\bot}v|^2}{r}\right)^{1/2}\left(\int_{r\leq r_0}|B_{\tau}|^2r|v|^2\right)^{1/2}\\ \notag
&+2\tau\beta\left(\int_{r> r_0}\log r\frac{|\D^{\bot}v|^2}{r}\right)^{1/2}\left(\int_{r> r_0}r \log r|B_{\tau}|^2|v|^2\right)^{1/2}\\ \notag
&+4\tau\beta\left(\int_{r> r_0}\frac{|\D^{\bot}v|^2}{r}\right)^{1/2}\left(\int_{r> r_0}r|B_{\tau}|^2|v|^2\right)^{1/2} \notag
\end{align}

Let us first study the case $r\leq r_0$ of all the potential terms. By Assumption \ref{ass1} and by (\ref{c4}),  the integrals in $r\leq r_0$ of the RHS of (\ref{c4}) can be bounded by
\begin{align}
2\tau\beta\left(\frac{\log r_0}{2}+1\right)\left[\mu_1\int_{r\leq r_0}\frac{|v|^2}{r^3}+2\nu_1\left(\int_{r\leq r_0}\frac{|\D^{\bot}v|^2}{r}\right)^{\frac{1}{2}}\left(\int_{r\leq r_0}\frac{|v|^2}{r^3}\right)^{\frac{1}{2}}\right].
\end{align}
The idea is to hide these terms with the first term in (\ref{c3}) following the approach in \cite{F}, section 3 or \cite{Z}. For simplicity, we denote
$$a:=\left(\int_{r\leq r_0}\frac{|\D^{\bot}v|^2}{r}\right)^{\frac{1}{2}}, \quad \qquad b:=\left(\int_{r\leq r_0}\frac{|v|^2}{r^3}\right)^{\frac{1}{2}}.$$
Hence, we need 
$$[(d-1)(d-2)-\mu_1]b^2+2a^2-2\nu_1ab>0,$$
with follows from the smallness conditions of the constants $\mu_1, \nu_1$ in Assumption \ref{ass1}.

Similarly, and also by Assumption \ref{ass1} and (\ref{c4}), for $r>r_0$, the terms containing $\log r$ in the RHS of (\ref{c3}) can be hidden with the terms containing $\frac{\log r}{r}$ in the left hand of (\ref{c3}). Finally, the rest of the potential terms in the RHS of (\ref{c3}) can be hidden with 
\begin{equation}\label{hide}
3\tau\beta\int_{r>r_0}\frac{|\D^{\bot}v|^2}{r}+2\tau^3\beta^3\int_{r>r_0}\frac{|v|^2}{r}.
\end{equation}
Since $d\geq 3$ we have the positivity in all the terms  of the LHS of (\ref{c3}) and (\ref{carleman}) follows.

\textit{For Assumption \ref{ass2}. } Since for $r\leq r_0$ the conditions are the same as the ones in Assumption \ref{ass1}, we will just study the case $r>r_0$. We go back to the RHS of (\ref{c3}) and we will start working with the terms containing $\log r$. Since the imaginary part is bounded by the absolute value, by Assumption \ref{ass2} we have that those terms are upper bounded by
\begin{align*}
&2\tau\beta c\int_{r>r_0} \log r \frac{|\D^{\bot}v|}{r^{1/2}}\frac{|v|}{r^{1/2+\mu}}+\tau\beta c\int_{r>r_0}\log r\frac{|v|^2}{r^{1+\mu}}.
\end{align*}
Now by using the Cauchy-Schwarz inequality and the property $2ab\leq a^2+b^2$ we estimate the previous term by
\begin{align}\label{sr1}
\frac{\tau\beta c}{r_0^{\mu}}\int_{r>r_0} \log r \frac{|\D^{\bot}v|^2}{r}+\frac{2\tau\beta c}{r_0^{\mu}}\int_{r>r_0}\log r\frac{|v|^2}{r}.
\end{align}
We hide (\ref{sr1}) with the terms containing $\frac{\log r}{r}$ in the LHS of (\ref{c3}). For this we need to have 
$$2\tau\beta-\frac{\tau\beta c}{r_0^{\mu}}>0,\qquad 2\tau^3\beta^3-\frac{2\tau\beta c}{r_0^{\mu}}>0,$$
which hold for $r_0$ large enough. 

To finish with the proof we will hide the remaining potential terms in the RHS of (\ref{c3}) in a similar way. Here also we will use assumption \ref{ass2} and the same properties as in the previous estimate so we get
\begin{align}\label{sr2}
&4\tau\beta\int_{r>r_0} B_{\tau}\cdot \D^{\bot}v \bar{v}+2\tau\beta\int_{r>r_0}(\partial_r V)_- |v|^2\\ \notag
&\leq 4\tau\beta\int_{r>r_0} \frac{|\D^{\bot}v|}{r^{1/2}}\frac{|v|}{r^{1/2+\mu}}+2\tau\beta c\int_{r>r_0}\frac{|v|^2}{r^{1+\mu}}\\ \notag
&\leq \frac{2\tau\beta c}{r_0^{\mu}}\int_{r>r_0}\frac{|\D^{\bot}v|^2}{r}+\frac{4\tau\beta c}{r_0^{\mu}}\int_{r>r_0}\frac{|v|^2}{r}.
\end{align}
The RHS of the previous estimate can be hidden with the terms (\ref{hide}), which are from the LHS of (\ref{c3}), if 
$$4\tau\beta-\frac{2\tau\beta c}{r_0^{\mu}}>0, \qquad 2\tau^3\beta^3-\frac{4\tau\beta c}{r_0^{\mu}}>0,$$
which is true for $r_0$ large enough. 

Combining all these computations and writing the integrals in terms of $u$ we obtain estimate (\ref{carleman}), which completes the proof.

\end{proof}

\begin{remark}\label{extra}
It is worth pointing out that the Carleman estimate (\ref{carleman}) is also true under weaker assumptions on the electric potential. In fact, if we require that 
\begin{equation*}
\int_{r\leq r_0} (\partial_r V)_-|u|^2\leq \mu_1 \int_{r\leq r_0}\frac{|u|^2}{r^3} +\int_{r\leq r_0} (\partial_r V)_+|u|^2
\end{equation*} 
and
\begin{equation*}
\int_{r> r_0} (\log r+2)(\partial_r V)_-|u|^2\leq \mu_2 \int_{r> r_0}\frac{|u|^2}{r} +\int_{r> r_0} (\log r+2)(\partial_r V)_+|u|^2
\end{equation*} 
in Assumption \ref{ass1} or
\begin{equation*}
\int_{r> r_0} (\log r+2)(\partial_r V)_-|u|^2\leq c \int_{r> r_0}\frac{|u|^2}{r^{1+\mu}} +\int_{r> r_0} (\log r+2)(\partial_r V)_+|u|^2
\end{equation*} 
in Assumption \ref{ass2}, then Theorem \ref{procarleman} follows.

\end{remark}

If we take $V=0$ then we get the following Carleman estimate for $H_A^{\lambda}$.

\begin{cor}\label{coroexpo}
Under assumptions of Theorem \ref{procarleman} with $V=0$, we obtain
\begin{align}\label{corcarleman}
&\tau \int e^{2\tau \varphi}\frac{|u|^2}{|x|^3}+\tau\int_{r > r_0} e^{2\tau \varphi}\frac{\log |x|}{|x|^3}|u|^2 \\
&+\tau^2(1+\tau)\int_{r>r_0} e^{2\tau \varphi}(\log |x|+1)\frac{|u|^2}{|x|}+\tau\lambda\int_{r>r_0} e^{2\tau \varphi}\frac{|u|^2}{|x|}\notag\\ 
&-\frac{\tau\beta(d-2)}{2r_0^2}\int_{r=r_0}e^{2\tau \varphi}|u|^2+\frac{\tau\beta}{r_0}\Re \int_{r=r_0}e^{2\tau \varphi}\D^ru \bar{u}\notag\\
& \leq C \int e^{2\tau \varphi}|\D^2 u +\lambda u|^2, \notag
\end{align}
where $C>0$ is independent of $\tau$.

\end{cor}

\subsection{Polynomial weight}\label{policarleman}

We show a Carleman estimate that generalizes the one given by H\"ormander \cite{Ho} (Proposition 14.7.1) to the magnetic case.

\begin{thm}\label{carlemanpoly}
Let $d \geq 3$, $\tau >0$, $\lambda > 0$. Assume that $B_\tau \equiv 0$ and $V$ satisfies (\ref{assuV}). Then for all $u \in C_0^\infty(\Rd)$ holds
\begin{equation}\label{polycarleman}
4\tau\lambda \int |x|^{2\tau} |u|^2 \leq \int |x|^{2\tau +2} \left|\D^2 u+Vu +\lambda u \right|^2.
\end{equation} 
\end{thm}

\begin{proof}
We follow the approach of Theorem \ref{procarleman}. Let $\varphi(x) = \log |x|$ so that $v=|x|^\tau u$. Then identity (\ref{equ1}) becomes into
\begin{align}
|x|^{\tau} \left(\D^2 u + V u + \lambda u \right) & = \D^2 v + V v + \lambda v + \tau^2 |x|^{-2}v\notag\\ 
& - \frac{2\tau}{|x|^{2}}\left(x \cdot \D v + \frac{(d-2)}{2}v \right)\notag
\end{align}
and 
\begin{align}
& L_1 v = (\D^2 + V + \lambda  + \tau^2 |x|^{-2}) v \notag\\
& L_2  v= - \frac{2\tau}{|x|^2}\left(x\cdot \D v + \frac{(d-2)}{2}v \right).\notag
\end{align}

Let $M$ denote the integral on the right-hand side of the inequality (\ref{polycarleman}), obtaining
\begin{align}
M & := \int |x|^2 |L_1 v + L_2 v|^2 \notag\\
& = \int |x|^2 |L_1 v|^2 + \int |x|^2 |L_ 2 v|^2 + 2\Re \int |x|^2 L_1 v \overline{L_2 v}. \notag
\end{align}
In particular, it yields
\begin{align}\label{mre1}
2\Re \int |x|^2 L_1 v \overline{L_2 v} \leq M.
\end{align}

We can now proceed analogously to the proof of Theorem \ref{procarleman}. By integration by parts techniques we get
\begin{align}
2\Re \int |x|^2  L_1 v \overline{L_2 v} & = 4\tau\lambda \int |v|^2 + 2\tau \int \left[ \partial_r (rV) + V \right] |v|^2\label{mre2}\\ 
& - 4\tau \Im \int |x| B_\tau \cdot \D v \bar{v}\notag.
\end{align}

Now, writing $\partial_{r}(rV) + V = [\partial_{r}(rV) + V]_{+} - [\partial_{r}(rV) + V]_{-}$ and using (\ref{assuV}), $B_{\tau} \equiv 0$, (\ref{mre1}) and (\ref{mre2}) we get (\ref{polycarleman}), and the proof is complete.

\end{proof}

\begin{remark}
In the polynomial case we consider $M = \int |x|^2 |L_1 v + L_2 v|^2$ and not $M = \int |L_1 v + L_2 v|^2$ as in the proof of Theorem \ref{procarleman}. The argument in Theorem \ref{procarleman} does not work properly with $\varphi = \log |x|$ because we will obtain
\begin{align}
2\Re \int L_1 v \overline{L_2 v} & = -[2\tau^3 + \tau(2d-5)]\int \frac{|v|^2}{|x|^4} + 2\tau \int \frac{|\D^{\bot}v|^2}{|x|^2} - 2\tau \int \frac{|\D^r v|^2}{|x|^2},
\end{align}
and the positivity of the terms breaks down. We solve this problem including the weight $|x|^2$ in the definition of $M$. However, this does not give any gradient term $\D v$ in the left hand-side of the inequality (\ref{polycarleman}). Hence, we are not able to hide the potential term and we need to assume that $B_\tau \equiv 0$.
\end{remark}

Taking $V =0$ in the above result, we get the analogue estimate as H\"ormander does for the magnetic Laplacian.

\begin{cor}\label{coropoly}
Let $d\geq 3$, $\tau >0$, $\lambda >0$. Assume that $B_{\tau} \equiv 0$. Then for all $u \in C_{0}^{\infty}(\Rd)$ yields
\begin{equation}\label{polycarleman1}
4\tau\lambda \int |x|^{2\tau} |u|^2 \leq \int |x|^{2\tau +2} \left|\D^2 u +\lambda u \right|^2.
\end{equation} 
\end{cor}

\section{Unique continuation}\label{uc}

In this section we prove the main results of this paper. We show that $H_{A,V}^{\lambda}$ has a unique continuation property by using the Carleman estimates with exponential weights (Theorem \ref{procarleman}, Corollary \ref{coroexpo}). The strong unique continuation property for $H_{A,V}^{\lambda}$ is proved by estimates involving polynomial weights (Theorem \ref{carlemanpoly}, Corollary \ref{coropoly}) . We follow the same scheme as H\"ormander \cite{Ho} (Theorem 7.4.2) for both theorems. 

First we introduce some preliminaries that will be useful for the proofs. Let $u$ be a solution of the equation $\D^2 u + V u + \lambda u = 0.$ Choose $\chi \in C_{0}^{\infty}$ such that $\chi(x)=1$ when $|x|\leq 1$ and $\chi(x)=0$ when $|x|>2$. Set 
\begin{align}\label{urR}
u_{\rho,R}(x) = \chi\left( \frac{x}{R}\right)u_\rho(x) \quad \text{where} \quad u_{\rho}(x)= \left(1-\chi\left(\frac{x}{\rho} \right) \right)u(x),
\end{align}
for $\rho$ small and $R$ large.

The main idea of the proofs is to apply the corresponding Carleman estimate to $u_{\rho,R}$ and let $R \to \infty$.

Let us next show some properties related to $u_{\rho,R}$. By (\ref{urR}) it yields
\begin{align}\label{lapmag}
(\D^2 + V + \lambda)u_{\rho,R}  & = \chi\left(\frac{x}{R}\right)(\D^2 + V + \lambda)u_{\rho} + 2\frac{\chi'}{R}\frac{x}{|x|}\cdot\D u_\rho \\
& + \left[\frac{\chi'(d-1)}{|x|R} + \frac{\chi''}{R^2} \right]u_\rho.\notag
\end{align}
Note that for any $R$, 
\begin{equation}\label{ur1}
u_{\rho,R} \geq u_\rho.
\end{equation}
In addition, since $\chi'$, $\chi''$ are defined in the annulus $R \leq |x| \leq 2R$, by (\ref{lapmag}) it follows that for a given function $\varphi$
\begin{align}\label{ur2}
\int e^{2\tau\varphi}\left|(\D^2 + V + \lambda)u_{\rho,R} \right|^2 & \leq  \int_{|x| <2R} e^{2\tau\varphi}\left|(\D^2 + V + \lambda)u_\rho \right|^2\\
& + \int_{R < |x| < 2R} e^{2\tau\varphi}\left[ \frac{|\D u_\rho|^2}{R^2} +
\frac{|u_\rho|^2}{R^4}\right].\notag
\end{align}

Finally, observe that
\begin{align}\label{ur3}
(\D^2 + V + \lambda) u_\rho & = \left(1-\chi\left(\frac{x}{\rho}\right)\right)(\D^2 + V + \lambda )u - \frac{2\chi'}{\rho}\frac{x}{|x|}\cdot \D u\\
& - \left(\frac{\chi''}{\rho^2} + \frac{\chi'(d-1)}{|x|\rho} \right)u\notag.
\end{align}

Now we are ready to show the unique continuation theorems.

\subsection{Proof of Theorem \ref{UCP}}
We begin by analyzing the boundary terms of the estimate (\ref{carleman}) in order to get the positivity of its LHS. Since $u$ is solution of $\D^2 u + V u +\lambda u = 0,$ multiplying this equation by $e^{2\tau\varphi}\bar{u}$ and integrating over $r> r_0$, yields
\begin{align}\notag
\int_{r>r_0}e^{2\tau\varphi}\D^2 u \bar{u} + \int_{r>r_0} e^{2\tau\varphi} V|u|^2 + \lambda\int_{r>r_0}e^{2\tau\varphi}|u|^2 = 0.
\end{align}
Integrating by parts the first term of the above identity, taking the real parts and inserting it in the identity, we get
\begin{align}\label{iu1}
\Re\int_{r=r_0} e^{2\tau\varphi} \D^r u\bar{u}  =& -\int_{r>r_0} e^{2\tau\varphi}|\D u|^2 +\int_{r>r_0} e^{2\tau\varphi} V|u|^2 \\
& +\lambda \int_{r>r_0} e^{2\tau\varphi} |u|^2 +\tau(d-1) \int_{r>r_0} e^{2\tau\varphi}\varphi' \frac{|u|^2}{r}\notag\\
& + \tau\int_{r>r_0} e^{2\tau\varphi} \varphi''|u|^2 +2\tau^2 \int_{r>r_0} e^{2\tau\varphi} (\varphi')^2 |u|^2\notag\\ 
&  +\tau\int_{r=r_0} \varphi'(r_0)e^{2\tau\varphi} |u|^2.\notag
\end{align}

Let us insert (\ref{iu1}) into (\ref{carleman}). Observe  that the regularity assumption $u\in \h^1$ is enough to justify all the integral terms in (\ref{iu1}) and (\ref{carleman}). Now we show the positivity of the terms in the resulting estimate. Taking $\tau > \frac{(d-2)}{\beta r_0(2+\log r_0)}$ and replacing $\varphi'$ using (\ref{fi}), we get
$$
\frac{\tau\beta}{r_0}\left[\tau\varphi'(r_0) -\frac{(d-2)}{2r_0} \right]\int_{r=r_0}e^{2\tau\varphi} |u|^2 > 0.
$$
In addition, since $r_0 \geq 1$, we have
$$
\tau\beta \left(1-\frac{1}{r_0} \right)\int_{r>r_0} e^{2\tau\varphi}|\D u|^2 > 0.
$$
Finally, writing $V = V_+ - V_-$, by (\ref{fi}) and (\ref{vbarrixa}), it is easy to check that
$$
\int_{r>r_0} e^{2\tau\varphi} \left[2\tau^2(\varphi')^2 + \frac{\tau(d-1)\varphi'}{r} + \tau\varphi'' - V_{-} \right] |u|^2 \geq 0.
$$

Hence, estimate (\ref{carleman}) gives
\begin{align}\label{carlemanuc1}
\tau \int e^{2\tau \varphi}\frac{|u|^2}{r^3}+&\tau^2(1+\tau)\int_{r>r_0} e^{2\tau \varphi}(\log r+1)\frac{|u|^2}{r}\\
& \leq C \int e^{2\tau \varphi}|\D^2 u +V u +\lambda u|^2. \notag
\end{align}
This estimate is the one that we use to prove the result.

We first work under hypotheses (i). Let us apply the Carleman estimate (\ref{carlemanuc1}) to $u_{\rho,R}$. In particular, it holds
\begin{align}\label{*}
\tau\int  e^{2\tau\varphi} \frac{|u_{\rho,R}|^2}{r^3} &  \leq C \int e^{2\tau\varphi} \left|\left(\D^2 + V + \lambda \right) u_{\rho,R} \right|^2.
\end{align}

Now, by (\ref{ur1}) it follows that the left hand-side of (\ref{*}) can be lower bounded by
\begin{equation}
\nu \tau \int e^{2\tau\varphi} \frac{|u_{\rho}|^2}{r^3},
\end{equation}
for some $\nu >0$.

Concerning the right hand-side of (\ref{*}), first apply (\ref{ur2}) and let us analyze its right hand-side. Take $R > r_0$ and $\rho < \frac{1}{2}$ so that $R > 2\rho$ and observe that $u_\rho = u$ if $x > 2\rho$. Hence, by using that $\varphi \leq \frac{\beta}{2}|x|(\log |x| + 2)$ when $|x| > r_0$, we have
\begin{align}
\int_{R < |x| < 2R} e^{2\tau\varphi}\left[ \frac{|\D u_\rho|^2}{R^2} +
\frac{|u_\rho|^2}{R^4}\right] &\leq \frac{e^{2\beta\tau R(\log (2R) + 2)}}{R^2}\int_{|x| > R}  (|\D u|^2 + |u|^2)\notag,
\end{align}
which by (\ref{exp}) tends to zero as $R \to \infty$.

To deal with the $(\D^2 + V + \lambda)u_\rho$ term, we use (\ref{ur3}). Note that in this case $\chi'$, $\chi''$ are defined on $\rho\leq |x| \leq 2\rho$. Take $\rho$ such that $2\rho < 1$. Thus by (\ref{fi}) and since $u\in \mathcal{H}^1$, yields
\begin{align}
\int_{|x|<2R} e^{2\tau\varphi}\left|(\D^2 + V +\lambda)u_\rho \right|^2 & \leq  \frac{1}{\rho^4} \int_{\rho\leq |x| \leq 2\rho} e^{2\tau\varphi}(|\D u|^2 + |u|^2)\notag\\
& \leq \frac{Ce^{2\tau\beta\rho(2 + \log r_0)}}{\rho^4}\notag.
\end{align}

Therefore, letting $R \to \infty$, we deduce
\begin{equation}
\int e^{2\tau\varphi} \frac{|u_\rho|^2}{r^3} \leq \frac{Ce^{2\tau\beta\rho(2 + \log r_0)}}{\tau \rho^4}.
\end{equation}
By (\ref{fi}) and since $2\rho < 1 \leq r_0$, we get
\begin{align}
\int e^{2\tau\varphi} \frac{|u_\rho|^2}{r^3} & > \int_{2\rho < |x| < r_0} e^{2\tau\varphi} \frac{|u|^2}{r^3} + \int_{|x|> r_0} e^{2\tau\varphi}\frac{|u|^2}{r^3} \notag\\
& > \min \left\{ e^{4\tau\beta\rho + 2\tau\beta\rho\log r_0}, e^{\tau\beta (r_0 + 1) + \tau\beta r_0\log r_0} \right\}\int_{|x|> 2\rho} \frac{|u|^2}{r^3}\notag\\
& = e^{2\tau\beta\rho (2+\log r_0)}\int_{|x|> 2\rho} \frac{|u|^2}{r^3}\notag,
\end{align}
which implies that
\begin{align}
\int_{|x|> 2\rho} \frac{|u|^2}{r^3} \leq \frac{C}{\tau\rho^4}.
\end{align}
Letting $\tau \to \infty$, we have
\begin{equation}
\int_{|x|> 2\rho} \frac{|u|^2}{r^3} = 0.
\end{equation}
Thus we deduce that 
$$
\Vert u\Vert_{L^{2, -\frac{3}{2}}(|x|>2\rho)}=0.
$$
Here $L^{2, -\frac{3}{2}}(|x|>2\rho)$ denotes the weighted space
$$
L^{2,\delta}(G) = \{ f : (1+|x|^2)^{\delta/2}f \in L^{2}(G)\}
$$
with $\delta= -3/2$ and $G= \{|x| > 2\rho\}$. As a consequence, it may be concluded that $u=0$ when $|x|> 2\rho$. Finally, since $\rho$ is any small positive number, we get that $u \equiv 0$ everywhere.

We now turn to the condition (ii). In this case we apply the Carleman estimate (\ref{carlemanuc1}) with $V=0$ to $u_{\rho,R}$ with $\rho$ such that $2\rho < 1$. Thus we obtain
\begin{align}\label{**}
&\tau\int_{|x|\leq r_{0}} e^{2\tau\varphi} \frac{|u_{\rho,R}|^2}{r^3} + \tau^2(1+\tau)\int_{|x|>r_{0}} e^{2\tau\varphi} (1+\log r)\frac{|u_{\rho,R}|^2}{r} \\
&\leq C \int e^{2\tau\varphi} \left| (\D^2 + \lambda) u_{\rho,R} \right|^2.\notag
\end{align}

The left-hand side of the inequality can be handled in the same way as above. We get that it is lower bounded by
\begin{equation}
\tau\int_{|x|\leq r_{0}} e^{2\tau\varphi}\frac{|u_\rho|^2}{r^3} + \tau^2(\tau + 1)  \int_{|x|> r_{0}} e^{2\tau\varphi} (1+\log r)\frac{|u_{\rho}|^2}{r}.
\end{equation}

To deal with the right-hand side of (\ref{**}) we use (\ref{ur2}) and we proceed as in (i). The only difference is the analysis of the term $(\D^2 + \lambda)u_{\rho}$. By (\ref{ur3}) with $V=0$, using that $(\D^2 + \lambda)u = -V u$, (\ref{Vucp}) and $u\in \mathcal{H}^1$, it follows that
\begin{align}
&\int_{|x|< 2R} e^{2\tau\varphi} \left|(\D^2 + \lambda)u_{\rho}\right|^2  \leq \int_{|x|< 2R} e^{2\tau\varphi}|V u_{\rho}|^2 + \frac{Ce^{2\tau\rho(2\beta + \beta\log r_0)}}{\rho^4} \notag\\
& \leq \gamma_{1}\int_{|x|\leq r_{0}} e^{2\tau\varphi} \frac{|u_\rho|^{2}}{r^3} + \gamma_{2}\int_{|x|> r_{0}} e^{2\tau\varphi}(1+\log r) \frac{|u_\rho|^2}{r} +\frac{Ce^{2\tau\rho(2\beta + \beta\log r_0)}}{\rho^4}\notag.
\end{align}

Hence, taking $\tau$ large enough such that $\tau - \gamma_1 >0$ and $\tau^2(\tau +1) - \gamma_2 >0$, since $2\rho<1 \leq r_{0}$, we get
\begin{align}
\int_{2\rho < |x| \leq r_{0}} e^{2\tau\varphi}\frac{|u|^2}{r^3} + \int_{|x|> r_0} e^{2\tau\varphi}\frac{|u|^2}{r}\leq \frac{Ce^{2\tau\rho(2\beta + \beta\log r_0)}}{\tau \rho^{4}}.
\end{align}
Now using that $\frac{1}{r} > \frac{1}{r_0}$ in the region $\{2\rho < |x| \leq r_0 \}$ and $1 \geq \frac{1}{r_0}$, it may be concluded that
\begin{equation}
\int_{|x| > 2\rho} \frac{|u|^2}{r} \leq \frac{C r_{0}^2}{\tau \rho^4}.
\end{equation}
Thus letting $\tau \to \infty$, we have
\begin{equation}
\int_{|x|> 2\rho} \frac{|u|^2}{r} = 0
\end{equation}
and we deduce that 
\begin{equation}
\Vert u \Vert^2_{L^{2, -\frac{1}{2}}(|x|> 2\rho)} = 0.
\end{equation}
This gives $u = 0$, which completes the proof.


\subsection{Proof of Theorem \ref{(SUCP)}}

This follows in much the same way as in the proof of Theorem \ref{UCP}. Thus we only give the main ideas of the proof.

When $V$ holds condition (\ref{assuV}), we apply the Carleman estimate (\ref{polycarleman}) to $u_{\rho, R}$. Then by (\ref{pol}), analysis similar to that in the first part of the proof of Theorem \ref{UCP} gives
\begin{equation}
\lambda\tau \int |x|^{2\tau}|u_{\rho}|^2 \leq C \rho^{2\tau -2}.
\end{equation}
Hence, it yields
\begin{equation}
\int_{|x|> 2\rho} |u|^{2} \leq \frac{C}{\tau\lambda \rho^2},
\end{equation}
and it can be deduced that $u=0$ everywhere.

The same conclusion can be drawn for the case when the potential $V$ satisfies (\ref{Vsucp}), by using the Carleman estimate (\ref{polycarleman1}), condition (\ref{pol}) for the solution $u$ and the same reasoning as in the second part of the proof of Theorem \ref{UCP}. The details are left to the reader.

\section*{Acknowledgements} The author gratefully acknowledges the many helpful suggestions of Luis Escauriaza and Luis Vega during the preparation of the paper. The first author was supported in part by MTM2011-24054 and IT641-13. The second author has received funding from European Research Council under the European Union's Seventh Framework Programme (FP7/2007-2013) / ERC grant agreement n. 267700.


\begin{thebibliography}{99}

\bibitem{BFRV} Barcel\'o,  J. A., Fanelli, L. , Ruiz, A., Vilela, M. (2013). \emph{A priori estimates for the Helmholtz equation with electromagnetic potentials in exterior domains}, Proc. Roy. Soc. Edinburgh Sect. A 143 , no. 1, 1-19.

\bibitem{BKRS}  Barcel\'o, B., Kenig, C.E., Ruiz, A., Sogge, C.D. (1988). \emph{Weighted Sobolev inequalities and unique continuation for the Laplacian plus lower order terms}, Illinois J. Math. 32 , 230-245.

\bibitem{BVZ} Barcel\'o, J. A., Vega, L., Zubeldia, M. (2013). \emph{The forward problem for the electromagnetic Helmholtz equation with critical singularities}, Adv. Math. 240, 636-671.

\bibitem{PLNL} D'Ancona, P., Fanelli, L., Vega, L., Visciglia, N. (2010). \emph{Endpoint Strichartz estimates for the magnetic Schr\"odinger equations}, J. Funct. Analysis 258, 3227-3240.

\bibitem{F} Fanelli, L. (2009). \emph{Non-trapping magnetic fields and Morrey-Campanato estimates for Schr\"odinger operators}, J. Math. Anal. Appl. 357, 1-14.

\bibitem{FFFP} Fanelli, L.,  Felli,V., Fontelos, M. A., Primo, A. (2012). \emph{Time decay of scaling critical electromagnetic Schr\"odinger flows}, \url{arXiv:1203.1771}.

\bibitem{FG} Fanelli, L., Garc\'ia, A. (2011). \emph{Counterexamples to Strichartz estimates for the magnetic Schr\"odinger equation}, Commun. Contemp. Math. 13, no. 2, 213-234.

\bibitem{FV} Fanelli, L., Vega, L. (2009). \emph{Magnetic virial identities, weak dispersion and Strichartz inequalities}, Math. Ann. 344, 249-278.

\bibitem{GL1} Garofalo, N., Lin, F.H. (1986). \emph{Monotonicity properties of variational integrals, $A_p$ weights and unique continuation}, Indiana Univ. Math. J. 35, 245-268.

\bibitem{GL2} Garofalo, N., Lin, F.H. (1987). \emph{Unique continuation for elliptic operators: A geometric-variational approach}, Comm. Pure Appl. Math. 40, 347-366.

\bibitem{Ho} H\"ormander, L. (1983). \emph{The Analysis of Linear Partial Differential Operators}, vols. I,II, Springer-Verlag, Berlin.

\bibitem{KT} Koch, H., Tataru, D. (2009). \emph{Carleman estimates and unique continuation for second order parabolic equations with nonsmooth coefficients}, Comm. P.D.E. 34, no.4-6, 305-366.

\bibitem{Ku1} Kurata, K. (1993). \emph{A unique continuation theorem for uniformly elliptic equations with strongly singular potentials}, Comm. in P.D.E. 18, 1161-1189.

\bibitem{Ku2} Kurata, K. (1997). \emph{A unique continuation theorem for the Schr\"dingier equation with singular magnetic field}, Proc. Amer. Math. Soc. 125, no. 3, 853-860.

\bibitem{Pa} Pan, Y. (1992). \emph{Unique continuation for Schr\"odinger operators with singular potentials}, Commun. Part. Diff. Eq. 17 (5-7), 953-965.

\bibitem{Re} Regbaoui, R. (1999). \emph{Unique continuation for differential equations of Schr\"odinger's type}, Comm. Anal. Geom. 7 (2), 303-323.

\bibitem{So} Sogge, C.D. (1990). \emph{Strong uniqueness theorems for second order elliptic differential equations}, Amer. J. Math. 112, 943-984.

\bibitem{Wo1} Wolff, T.H. (1990). \emph{Unique continuation for $|\Delta u| \leq V|\nabla u|$ and related problems}, Revista Math. Iberoamericana 6, 155-200.

\bibitem{xl}Lu, X. (2013). \emph{A unique continuation result for the magnetic Schr\"odinger operator}, \url{http://www.bcamath.org/documentos_public/archivos/publicaciones/1unique_continuation.pdf}.


\bibitem{Z} Zubeldia, M. \emph{Limiting absorption principle for the electromagnetic Helmholtz equation with singular potentials}, to appear in Proc. Roy. Soc. Edinburgh Sect. A.


\end{thebibliography}
\end{document}